\documentclass[11 pt]{article}
\usepackage[latin1]{inputenc}
\usepackage{amsmath,amsfonts,stmaryrd,hyperref,amsthm, enumitem, graphicx, amssymb}

\usepackage{fullpage}
\usepackage{authblk}

\newtheorem{thm}{Theorem}
\newtheorem{lemma}[thm]{Lemma}
\newtheorem{prop}[thm]{Proposition}

\newtheorem{conjecture}[thm]{Conjecture}
\theoremstyle{definition}

\newtheorem{rem}{Remark}

\newcommand{\qq}{\mathbf{q}}
\newcommand{\vv}{\mathbf{v}}
\newcommand{\ww}{\mathbf{w}}
\newcommand{\sst}{\mathbf{s}}
\newcommand{\ee}{\mathbf{e}}
\newcommand{\XX}{\mathbf{X}}
\newcommand{\YY}{\mathbf{Y}}

\newcommand{\qginf}{\qq_\infty^{(g)}}
\newcommand{\qgn}{\qq_{n}^{(g)}}
\newcommand{\QQgn}{\mathcal{Q}_n^{(g)}}
\newcommand{\pnalpha}{p_{n,\alpha}}

\title{On tessellations of random maps and the $t_g$-recurrence}

\author{Guillaume Chapuy%
\thanks{
	Support from the \emph{European Research Council}, grant ERC-2016-STG 716083 ``CombiTop'', 
from \emph{Agence Nationale de la Recherche}, grant number ANR~12-JS02-001-01 ``Cartaplus'', and from the City of Paris, grant ``\'Emergences Paris 2013, Combinatoire \`a Paris''. Part of this research was done while I was affiliated with the \small\it {\sc CRM, UMI CNRS 3457},
Universit\'e de Montr\'eal,
Canada.
Email:~{\tt guillaume.chapuy@irif.fr}.
}
}
\affil{\small\it {\sc IRIF, UMR CNRS 8243},
Universit\'e Paris-Diderot,
France.}

\begin{document}
\maketitle
\begin{abstract}
	We study the masses of the two cells in a Vorono\"i tessellation of the Brownian surface of genus $g\geq 0$ centered on two uniform random points. Making use of classical bijections  and asymptotic estimates for maps of fixed genus, we relate the second moment of these random variables to the Painlev\'e-I equation satisfied by the double scaling limit of the one-matrix model, or equivalently to the ``$t_g$-recurrence'' satisfied by the constants $t_g$ driving the asymptotic number of maps of genus $g\geq0$. This raises the question of giving an independent probabilistic or combinatorial derivation of this second moment, which would then lead to new proof of the $t_g$-recurrence.

	More generally we conjecture that for any $g\geq 0$ and $k\geq 2$, the  masses of the cells in a Vorono\"i tessellation of the genus-$g$ Brownian surface by $k$ uniform points follows a Dirichlet$(1,1,\dots,1)$ distribution.

\end{abstract}

\section{Introduction and results}

In this paper a \emph{map} is a graph embedded without edge crossings on a closed oriented surface, in such a way that the connected components of the complement of the graph, called \emph{faces}, are each homeomorphic to a disk. Loops and multiple edges are allowed, and maps are considered up to orientation preserving homeomorphisms. A map is \emph{rooted} if an edge is distinguished and oriented. The number $m_g(n)$ of rooted maps with $n$ edges on the surface of genus $g$ satisfies, for fixed $g\geq 0$ and $n\rightarrow \infty$:
\begin{eqnarray}\label{eq:mgn}
m_g(n) \sim t_g n^{\frac{5(g-1)}{2}} 12^n, \mbox{ for }t_g >0.
\end{eqnarray}
In genus $0$, this result follows from the exact formula $m_0(n)=\frac{2\cdot 3^n}{(n+2)(n+1)}{2n \choose n}$ due to Tutte~\cite{Tutte:census}. In higher genus, it was proved in~\cite{BC0} using generating functions.
A direct combinatorial interpretation of Tutte's formula for maps of genus $0$ was given by Cori and Vauquelin~\cite{CV} and much simplified by Schaeffer~\cite{Schaeffer:phd,ChassaingSchaeffer}. A combinatorial interpretation of~\eqref{eq:mgn} was given in~\cite{CMS} using the Marcus-Schaeffer bijection~\cite{MS} and further developped in~\cite{Chapuy:trisections}.

None of the methods just mentioned enable to say much about the sequence of constants $(t_g)_{g\geq 0}$ that appear in~\eqref{eq:mgn}, and indeed these references give explicit values only for very small values of~$g$. There is however a remarkable recurrence formula to compute these numbers, that we call the \emph{$t_g$-recurrence}. It is better expressed in terms of the numbers $\tau_g=2^{5g-2}\Gamma\left(\frac{5g-1}{2}\right)t_g$ and is given by:
\begin{eqnarray}\label{eq:tg}
\tau_{g+1} = \frac{(5g+1)(5g-1)}{3}\tau_{g}+\tfrac{1}{2}\sum_{g_1=1}^{g}\tau_{g_1}\tau_{g+1-g_1}, \ \ g \geq 0,
\end{eqnarray}
which enables to compute these numbers easily starting from $\tau_0=-1$. This result was first stated in mathematical physics in relation with the \emph{double scaling limit} of the one-matrix model, and  obtained via a non-rigorous scaling of expressions involving orthogonal polynomials  (we refer to~\cite[p201]{LZ} for historical references). A more algebraic  approach is based on the fact that the partition function of maps on surfaces, with infinitely many parameters marking vertex degrees, is a tau-function of the KP hierarchy. Going from the KP hierarchy to the recurrence \eqref{eq:tg} (or to an equivalent Painlev\'{e}-I ODE for an associated generating function) relies on a trick of elimination of variables that can be performed in different ways and whose generality is, as far as we know, yet to be fully understood (for the case of triangulations see \cite[Appendix B.]{mcfly} or \cite{GJ, tg} and for general maps see\cite{CC}).

The main contribution of this paper is to relate the recurrence \eqref{eq:tg} to another side of the story, namely the study of random maps and their scaling limits. We refer to~\cite{miermont:survey} for an introduction to this topic.
To state our main result we first need a few more definitions. A \emph{quadrangulation} is a map in which each face contains exactly four corners, {\it i.e.} is bordered by exactly four edge-sides. It is \emph{bipartite} if its vertices can be colored in black and white in such a way that there is no monochromatic edge.
 For each $n,g\geq0$, there is a classical bijection, due to Tutte, between rooted maps of genus $g$ with $n$ edges and rooted bipartite quadrangulations of genus $g$ with $n$ faces.

For $n,g\geq 0$, we let $\QQgn$ be the set of rooted bipartite quadrangulations of genus $g$ with $n$ faces (with the convention that there is a single quadrangulation with $0$ face, which has genus $0$, no edge, and two vertices). We let $\qgn \in_u \QQgn$ be a bipartite quadrangulation of genus $g$ with $n$ faces chosen uniformly at random (the notation $\in_u$ to denote a uniform random element of a set will be used throughout). We equip the vertex set of $\qgn$ with the graph distance, noted $\mathbf{d}_n$, and with the uniform measure, noted $\mathbf{\mu}_n$. This makes $\qgn\equiv (\qgn,\mathbf{d}_n,\mu_n)$ into a compact measured metric space. The set of (isometry classes of) such spaces is equipped with the Gromov-Hausdorff-Prokhorov (GHP) topology as in~\cite[Sec.~6]{Miermont:tessellations}. 
A \emph{Brownian surface of genus $g$}~\cite{Bettinelli1, Bettinelli, BettinelliMiermont} is a random compact measured metric space $(\qq_\infty, d_\infty, \mu_\infty)$ that is such that:
$$
(\qgn,\tfrac{1}{n^{1/4}}\mathbf{d}_n,\mu_n)
\longrightarrow
(\qginf, d_\infty, \mu_\infty),
$$
in distribution along some subsequence for the GHP topology.  The existence of Brownian surfaces of genus $g$ was proved in~\cite{Bettinelli1}, and their uniqueness for each $g \geq 1$ has been announced by Bettinelli and Miermont~\cite{BettinelliMiermont} (in genus $0$ the uniqueness is an important result proved independently by Miermont~\cite{Miermont:GH} and Le Gall~\cite{LG:GH}). However the uniqueness of the limit is not needed for our discussion since we will prove the convergence of all the observables we are interested in. Also note that some authors prefer to introduce an additional scaling factor $(8/9)^{1/4}$ to the distance, but this is irrelevant to our discussion so we choose to avoid it.

\begin{thm}\label{thm:obs1}
For $g\geq 0$,  let $(\qginf, d_\infty, \mu_\infty)$ be a Brownian surface of genus $g$. Let $\vv_1,\vv_2 \in \qq^{(g)}_\infty$ be chosen independently according to the probability  measure $\mu_\infty$, and let $\XX_g, 1-\XX_g$ be the masses of the corresponding cells in the nearest neighbour tessellation of $\qginf$ induced by $\vv_1$ and $\vv_2$, that is to say:
$$
\XX_g := \mu_\infty\Big(\big\{x\in \qginf,\ d_\infty(x,\vv_1)<\mathbf{d}_\infty(x,\vv_2) \big\}\Big).
$$
 Then the sequence of numbers $\tau_g=2^{5g-2}\Gamma\left(\frac{5g-1}{2}\right)t_g$ satisfies:
$$
\tau_{g+1} =  2(5g+1)(5g-1)\tau_{g}\cdot \mathbf{E}[\XX_{g}(1-\XX_g)] 
+\frac{1}{2}\sum_{g_1=1}^{g}\tau_{g_1}\tau_{g+1-g_1}, \ \ \ g\geq 0.
$$
\end{thm}
By comparing Theorem~\ref{thm:obs1} with the $t_g$-recurrence~\eqref{eq:tg} we immediately deduce:
\begin{thm}\label{cor:main}
For any $g\geq 0$, the random variable $\XX_g$ satisfies
$$\mathbf{E}[\XX_g(1-\XX_g)]=\frac{1}{6},$$ or equivalently
 $\mathbf{E}\XX_g^2=\frac{1}{3}.$
\end{thm}
The reader may find surprising that $\mathbf{E}\XX_g^2$ does not depend on $g\geq0$: indeed, although it is natural to expect that \emph{local} statistics of Brownian surfaces do not depend on the genus,  the nearest-neighbour tessellation depends \emph{globally}  of the metric space $\qginf$, that \emph{is} genus dependent (see~\cite{Bettinelli}).  This suggests that there exists a simple probabilistic or combinatorial interpretation of this mysterious fact, based on a symmetry of Brownian surfaces, but we have not been able to find it. %
We emphasize that, via Theorem~\ref{thm:obs1}, such an interpretation would provide  a proof of the $t_g$-recurrence independent of orthogonal polynomials, matrix models or integrable hierarchies. More generally, this unexpected property may hide some powerful symmetries of Brownian surfaces that could be useful for other means (even in genus $0$).

It is natural to ask if other moments of the variables $\XX_g$ or related random variables are computable, and in which way they depend on the genus.
Let $\vv_1,\vv_2,\dots,\vv_k$ be $k\geq 2$ points in $\qginf$ chosen independently at random according to the ``Lebesgue'' measure $\mu_\infty$. Let $(\YY_g^{(i:k)})_{1\leq i \leq k}$ be the masses of the $k$-nearest-neighbour cells induced by the $\vv_i$'s,
\textit{i.e.} for $1\leq i \leq k$ let 
\begin{align}\label{eq:defYY}
	\YY_g^{(i:k)}:=\mu_\infty \big\{ x\in \qginf, \ \forall j\in \{1,2,\dots,k\}\setminus\{i\}, \ d_\infty(x,\vv_i)<d_\infty(x,\vv_j)\big\}.
\end{align}
We note that $\XX_g=\YY_g^{(1:2)}$, so we could have used a single notation, but we prefer to keep the lighter notation $\XX_g$ for $\YY_g^{(1:2)}$ throughout the paper. The following result is similar to, and as mysterious as Theorem~\ref{cor:main}:
\begin{thm}\label{thm:3points}
For $g\geq 0$, the masses $\YY_g^{(1:3)}, \YY_g^{(2:3)}, \YY_g^{(3:3)}$ of the Vorono\"i cells induced by three independent Lebesgue distributed points in the Brownian surface of genus $g$ satisfy, for $g\geq 0$:
$$
\mathbf{E}[\YY_g^{(1:3)}\YY_g^{(2:3)}\YY_g^{(3:3)}] 
=\frac{1}{60}.
$$
\end{thm}
As we will see, the fact that this moment is computable reflects the existence of a combinatorial device  known as the ``trisection lemma''~\cite{Chapuy:trisections}. The fact that it does not depend on the genus, and that it coincides\footnote{if $U_1,U_2$ are two independent uniforms on $[0,1]$ and $I_1, I_2, I_3$ are the lengths of the three intervals they define,
then %
$\mathbf{E} (I_1I_2I_3)$ is the probability that five independent uniforms $U_1,U_2,V_1,V_2,V_3$ are ordered as $V_1<U_{1}\wedge U_2<V_2<U_1\vee U_{2}<V_3$, which is clearly equal to $\frac{2}{5!}=\tfrac{1}{60}$.} with the corresponding moment for a uniform three-division of the interval $[0,1]$, is as mysterious as for the previous result (or even more, since as we will see the computations leading to Theorem~\ref{thm:3points} are quite delicate and involve intermediate expressions that are complicated and magically become simpler at the last minute).

 We won't prove anything on higher moments or other values of~$k$ since we lack the tools to study them. However, numerical simulations suggest that the first joint moments of the random variables $(\YY_g^{(i:k)})_{1\leq i \leq k}$, for small values of $g$ and $k$, are close to what they are for a uniform partition of $[0,1]$ into $k$ intervals, {\it i.e.}  a Dirichlet$(1,1,\dots,1)$ random variable.
Theorems~\ref{thm:obs1} and~\ref{thm:3points} support this conjecture,  so we dare to state it explicitly:
\begin{conjecture}\label{conjecture}
	For $k\geq 2, g\geq 0$, let $\qginf \equiv (\qginf, d_\infty, \mu_\infty)$ be a genus $g$ Brownian surface and let $\vv_1,\dots,\vv_k$ be $k$ i.i.d. points sampled according to the distribution $\mu_\infty$. Then the random vector $(\YY_g^{(1:k)},\YY_g^{(2:k)},\dots, \YY_g^{(k:k)})$ defined by \eqref{eq:defYY} has the same law as the subdivision of the unit interval induced by $k-1$ independent uniform variables, {\it i.e.} a Dirichlet random variable of parameters $(1,1,\dots,1)$. In particular, for any $g\geq 0$, $\XX_g=\YY_g^{(1:2)}$ is uniform on $[0,1]$.
\end{conjecture}

\smallskip

To conclude this introduction, we emphasize that our main contribution relates the moment $\mathbf{E}\XX_g^2$ to the $g$-th step of the $t_g$-recurrence. In particular, the fact that  $\mathbf{E}\XX_0^2=1/3$ for the \emph{genus~$0$} Brownian map is only ``equivalent'' to the computation of the genus $1$ constant $t_1$, that can be performed by hand in several ways (and similarly, our proof of Theorem~\ref{thm:3points} for $g=0$ relies only on the value of the constants $t_1$ and $t_2$).
However, proving Conjecture~\ref{conjecture} even for $(g,k)=(0,2)$ would be interesting in itself. Readers familiar with Miermont's bijection~\cite{Miermont:tessellations} may try to approach this problem by exact counting of well-labelled 2-face maps (we have failed trying to do so\footnote{after a first version of this paper was made public, Emmanuel Guitter was able to apply this idea, and checked by a ``semi-rigourous'' calculation that for $(g,k)=(0,2)$ the law is indeed uniform~\cite{Guitter:voronoiBrownian}. Guitter's remarkable calculation is computer assisted and very heavy, moreover fully justifying all of the needed approximations seems technically difficult -- yet it strongly supports our conjecture. This method seems unfit to apply to \emph{general} values of $k$ (or $g$), and probably too heavy to apply to \emph{any other} value of $(g,k)$.}). One could also hope that in the future purely probabilistic methods (for example using the QLE viewpoint on the Brownian map~\cite{MillerSheffield}) will enable to determine the full law of~$\XX_0$ or even the law of the vector $(\YY_0^{(i:k)})_{1\leq i\leq k}$ for each $k$. 
In an opposite direction, we recall that the $t_g$-recurrence is only a ``shadow'' of the fact that the generating functions of maps satisfy a set of infinitely many partial differential equations called the KP hierarchy.
 It is natural to expect that other joint moments of the variables $\YY_g^{(i:k)}$, apart from the two cases we have been able to track, are related to these equations. This may lead to a way, based on integrable hierarchies, of approaching Conjecture~\ref{conjecture}.

\section{Proof of Theorem~\ref{thm:obs1}}

\subsection{Classical tools and notation}
\label{sec:prel1}

In this section we recall some classical ingredients from the toolbox of map enumeration, and we introduce some notation.

\noindent{\bf Asymptotic counting of maps by genus.}
For $g\geq 0$ we let $Q_g(z)$ be the generating function of rooted bipartite quadrangulations of genus $g$ by the number of faces, and we let $Q_g^\bullet(z)$ be the g.f. of the same objects where an additional vertex is pointed. We let $m_g(n) = [z^n] Q_g(z)$ and we use the same notation with $~^\bullet$. In what follows the notation $a(n)\sim b(n)$ means (classically) that $a(n)/b(n)\rightarrow 1$ when $n$ tends to infinity, while the notation $F(z)\sim G(z) $ means that both $F$ and $G$ have a unique dominant singularity at $z=\tfrac{1}{12}$, that both have a convergent Puiseux expansion in a neighbourhood of this point slit along the line $[\frac{1}{12},\infty)$, and that the  first singular term in this expansion is the same for both.

From \cite{BC0} (see also~\cite{CMS} for purely combinatorial proofs) we have for fixed $g\geq 0$:
$$
m_g(n) \sim t_g n^{\frac{5g-5}{2}} 12^n
\ \ , \ \ 
m^\bullet_g(n) 
= (n+2-2g) m_g(n) \sim t_g n^{\frac{5g-3}{2}} 12^n
$$
\begin{eqnarray}\label{eq:singQg}
Q^\bullet_g(z) \sim \Gamma(\tfrac{5g-1}{2}) t_g (1-12z)^{\frac{1-5g}{2}}
=
2^{2-5g}\tau_{g}
(1-12z)^{\frac{1-5g}{2}}
.
\end{eqnarray}

\begin{figure}
\begin{center}
\includegraphics[width=0.4\linewidth]{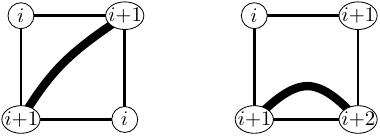}
	\caption{{\bf Schaeffer's rules.} The possible label variations around a face, up to rotation, are $(i,i+1,i,i+1)$ and $(i,i+1,i+2,i+1)$ where $i$ is the minimal label around the face. The Schaeffer rules indicate how to add a new edge (in bold) inside each face.
  }\label{fig:schaeffer}
\end{center}
\end{figure}

\noindent{\bf Miermont's bijection.}
We use $V(M), E(M)$ to denote respectively the vertex set and edge set of a map $M$.
A \emph{labelled map} of genus $g$ is a rooted map $M$ of genus $g$ equipped with a fonction $\ell: V(M)\rightarrow \mathbb{Z}$ such that for any edge $(u,v)$ of $M$ one has $\ell(u)-\ell(v)\in \{-1,0,1\}$. We consider these objects up to global translation of the labels.
A \emph{labelled one-face map (l.1.f.m.)} is a labelled map having only one face. We let $\mathcal{L}_n^{(g)}$ be the set of all (rooted) l.1.f.m. of genus $g$ with $n$ edges.

Miermont's bijection~\cite{Miermont:tessellations} plays an important role in this work. We briefly recall it for completeness, without proofs. 
For $k\geq 1$, a \emph{$k$-pointed bipartite quadrangulation} is a tuple 
$$(Q; s_1, s_2, \dots, s_k; \delta_1,\delta_2,\dots,\delta_k)$$
where $Q$ is a rooted bipartite quadrangulation,  $s_1, s_2, \dots, s_k$ are $k$ distinct vertices of $Q$ and where $\delta_1,\delta_2, \dots, \delta_k$ are integers, considered up to common translation. For the construction to be well defined, one needs the following two properties to hold
\begin{align}\label{eq:miermont1}
	\forall i,j, \ \ d(s_i,s_j) \equiv \delta_i - \delta_j \ \mod \ 2,
\\\label{eq:miermont2}
	\forall i\neq j \ \ d(s_i,s_j) > | \delta_i -\delta_j|,
\end{align}
where $d$ is the graph distance in $Q$.
The bijection of~\cite{Miermont:tessellations} proceeds as follows. We start by labelling each vertex $v$ of $Q$ by the number $\ell(v)$ defined by:
\begin{eqnarray}\label{eq:labelmiermont}
\ell(v) := \min_{i\in[1..k]} d(v,s_i)+\delta_i.
\end{eqnarray}
Bipartiteness and the hypotheses~\eqref{eq:miermont1}-\eqref{eq:miermont2} ensure that, up to rotation, the only  label sequences that can appear around a face are $(i,i+1,i,i+1)$ and $(i,i+1,i+2,i+1)$ for $i \in \mathbb{Z}$, see Figure~\ref{fig:schaeffer}. One then applies the ``Schaeffer rules'', that consist in adding a new edge inside each face of $Q$ according to the rules depicted in Figure~\ref{fig:schaeffer}. The labelled map $L$ associated to $Q$ is the labelled map consisting of all the new edges, and all the vertices of $Q$ different from the $k$ pointed vertices, with the labelling given by~\eqref{eq:labelmiermont}.

\begin{prop}[Miermont's bijection]\label{prop:bij}
	For each $g\geq 0$, $k\geq 1$, the preceding construction is a $1$ to $2$ mapping between $k$-pointed bipartite quadrangulations of genus $g$ satisfying~\eqref{eq:miermont1}-\eqref{eq:miermont2}, and labelled $k$-face maps of genus $g$, with faces numbered from $1$ to $k$. %
	Moreover, if $\ell_i$ is the minimal label in the $i$th face of the labelled map in some fixed labelling, then one has $\ell_i-\delta_i=\ell_j-\delta_j$ for all $1\leq i,j \leq k$.
\end{prop}

The converse bijection works as follows. Given a labelled $k$-face map $L$, let $\ell_i$ denote the minimum label in the $i$-th face for $1\leq i\leq k$, and add a new vertex of label $\ell_i-1$ inside that face (call it $s_i$). Then apply the \emph{closure} operation: join each corner of $L$ to the first corner with a strictly smaller label, counterclockwise around the face it belongs to (or join it to the newly added vertex in that face if no such corner exists). The map $Q$ formed by all the newly added edges, equipped with the $k$ marked points $s_1,\dots,s_k$  is a $k$-pointed quadrangulation, whose delays can be recovered up to translation by the last equality in Proposition~\ref{prop:bij}. The factor of $2$ comes from the necessary choice of a rooting convention, since $Q$ has twice as many edges as $L$. 
This description of the converse bijection makes clear that each corner $c$ of $L$ is canonically associated to an edge $e$ of $Q$. Since each corner belongs to a unique face, the $k$ faces of the map $L$ thus induce a partition of the edges of $Q$ into $k$ parts, $E(Q)=\cup_{a=1}^k E_a$ where $E_a$ is the set of edges of $Q$ lying in the $a$-th face of $L$. This partition can, roughly speaking, be understood as a Vorono\"i-like tessellation of the map $Q$, with respect to the ``delayed'' distance~\eqref{eq:labelmiermont}.

To make this statement precise, we proceed as in~\cite[Section 2.2]{Miermont:tessellations}. After labelling the vertices of $Q$ by~\eqref{eq:labelmiermont}, we orient each edge of $Q$ towards its endpoint of smaller label. From each edge $e$, the \emph{leftmost geodesic path starting from~$e$} is the oriented path that starts with $e$, and that when arriving at a new vertex continues with the leftmost available oriented edge around this vertex if there is such an edge, and stops otherwise. This path necessarily stops when it reaches one of the vertices $s_1,\dots,s_k$, and the converse bijection shows that for each $a \leq k$  the set of edges of $Q$ whose leftmost geodesic path ends at~$s_a$ is precisely the set $E_a$ defined above. For $a \leq k$, we also let $V_a$ (respectively, $V^{\neq}_a$) be the set of vertices of $Q$ where the minimum in~\eqref{eq:labelmiermont} is reached for $i=a$ (respectively, \emph{only} for $i=a$). 
We have~\cite[last eq. of Sec. 2.2]{Miermont:tessellations},
\begin{eqnarray}\label{eq:inclusionCells}
	V^{\neq}_a
	\subseteq\{s_a\} \cup \{e^-, e\in E_a\} 
	\subseteq
	V_a
\end{eqnarray}
where $e^-$ is the endpoint of $e$ with smaller label. This property makes the link between nearest-neighbour tessellations and labelled maps, and it will be crucial for us.

\smallskip
\noindent{\bf The Marcus-Schaeffer bijection.}
The Marcus-Schaeffer bijection (\cite{MS}, see also~\cite{CMS} for the version needed here) is the case $k=1$ of the Miermont bijection.
It is therefore a bijection:
\begin{eqnarray}\label{eq:MSbij}
\mathcal{Q}_n^{(g)\bullet} \longrightarrow \{\uparrow,\downarrow\} \times \mathcal{L}_n^{(g)},
\end{eqnarray}
where $\mathcal{Q}_n^{(g)\bullet}$ is the set of rooted bipartite quadrangulations of genus $g$ and $n$ faces  equipped with a pointed vertex.
It follows that $Q^\bullet_g(z)=2L_g(z)$ where $L_g(z)$ is the generating function  of rooted l.1.f.m. of genus $g$ by the number of edges.
Moreover, in genus $0$, rooted one-face maps are nothing but rooted plane trees, and a standard root-edge decomposition leads to the quadratic equation
$
L_0(z)= 1+3zL_0(z)^2,
$
from which we get the explicit formula:
\begin{eqnarray}\label{eq:valKernel}
 1-6zL_0(z) = \sqrt{1-12z}.
\end{eqnarray}

\subsection{The decomposition equation, Miermont's bijection, and proof of Theorem~\ref{thm:obs1}}

\begin{figure}
\begin{center}
\includegraphics[width=\linewidth]{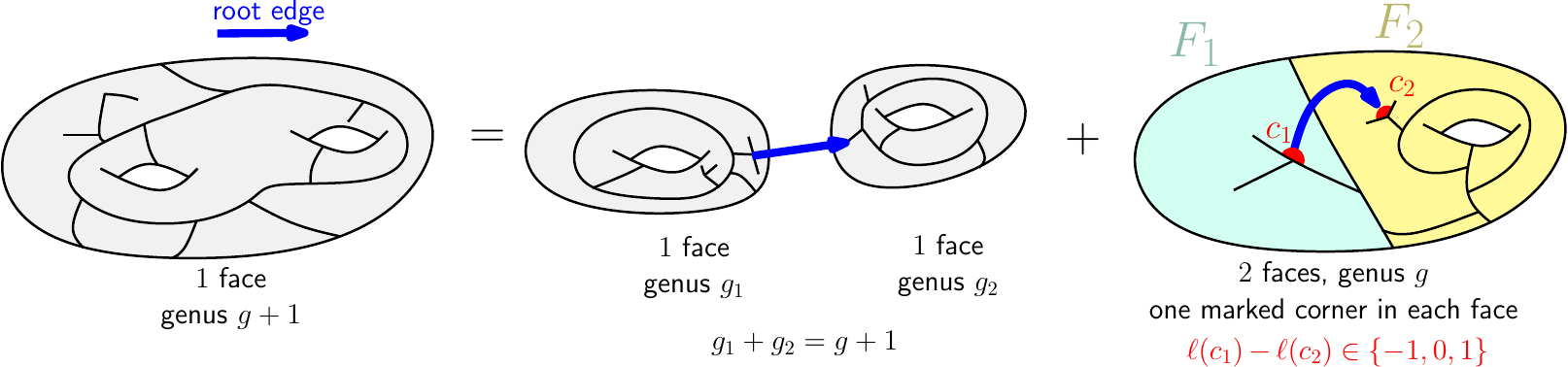}
\caption{Illustration of the decomposition leading to Equation~\eqref{eq:Tutte}}\label{fig:Tutte}
\end{center}
\end{figure}

We now come to the substance of this paper, which in concrete terms is to try to write an equation for the generating function of l.1.f.m. by root-edge decomposition, and see what happens. 

We fix $g\geq 0$, and we consider a l.1.f.m. $M$ of genus $g+1$. If we remove the root edge of this map, two things can happen (see Figure~\ref{fig:Tutte}):
\begin{itemize}[itemsep=0pt, topsep=0pt,parsep=0pt, leftmargin=20pt]
\item[(i)] we disconnect the map into two l.1.f.m. $M_1$ and $M_2$ whose genera sum up to $g+1$;
\item[(ii)] we do not disconnect the map; in this case we are left with a map $M'$ of genus $g$ with two faces. Each face of $M'$ carries a distinguished corner, and the labels of these two corners differ by $-1$, $0$, or $1$.
\end{itemize}

\noindent 
Translating this operation  into an equation for generating functions we obtain
\begin{eqnarray}\label{eq:Tutte}
L_{g+1}(z) 
= 
3z\sum_{g_1+g_2=g+1\atop g_1,g_2\geq 0}
L_{g_1}(z) 
L_{g_2}(z) 
+
z A_g(z)
\end{eqnarray}
where:
\begin{itemize}[itemsep=0pt, topsep=0pt, parsep=0pt, leftmargin=20pt]
\item[-]
 in the first term the factor of $3$ takes into account the choice of the increment of label along the root-edge in $\{-1,0,1\}$;
\item[-]  $A_g(z)$ is the generating function by the number of edges, of unrooted labelled two-face maps of genus $g$, with faces numbered $F_1,F_2$, such that the face $F_i$ contains a marked corner $c_i$ for $i\in\{1,2\}$, and that $|\ell(c_1)-\ell(c_2)|\leq 1$.
\end{itemize}

\noindent 
Objects counted by $A_g(z)$ are related to quadrangulations thanks to Miermont's bijection.
In the following discussion, %
we will show  how to use this observation to arrive informally at Lemma~\ref{lemma:convergenceCell} below, and why this implies Theorem~\ref{thm:obs1}. Details of the proof of Lemma~\ref{lemma:convergenceCell} are postponed to the next sections.

Let us consider an object counted by $[z^n] A_g(z)$. Let us fix the translation class of the labels by saying that the minimum label in face $F_1$ is zero, and let us call $\delta$ the minimum label in face $F_2$. Let $i_1\geq 0$ and $i_2\geq \delta$ be the labels of the two marked corners $c_1$ and $c_2$,  respectively, and recall that $i_1-i_2\in \{-1,0,1\}$.
 Miermont's bijection (Proposition~\ref{prop:bij} with $k=2$) puts this object in correspondence with a bipartite quadrangulation $Q$ of genus $g$ with two distinct marked vertices $s_1, s_2$ such that the quantity $d(s_1,s_2)+\delta$ is even, where $d$ is the graph distance in $Q$. Moreover,  the two corners $c_1$ and $c_2$ of the original two-face map are naturally associated to two edges $e_1$ and $e_2$ of the quadrangulation, and the second inclusion in \eqref{eq:inclusionCells}~ensures that  if $m_i$ is the endpoint of $e_i$ closer from $s_i$ in $Q$ for $i\in \{1,2\}$, one has: 
$$
d(s_1,m_1)=i_1 \ \ ,\ \ d(s_2,m_2)=i_2-\delta \ \ , \ \ 
d(s_2,m_1)\geq i_1-\delta \ \ ,\ \ d(s_1,m_2)\geq i_2.$$
These constraints imply 
\begin{align}\label{eq:crossed}
d(s_1,m_1) \leq d(s_1,m_2) -\epsilon\ \ , \ \ 
d(s_2,m_2) \leq d(s_2,m_1) +\epsilon,
\end{align}
where $\epsilon = i_2-i_1$.  Note that $|\epsilon|\leq 1$,
so loosely speaking the properties in~\eqref{eq:crossed} imply that, \emph{up to an error at most 1}, $s_i$ is (weakly) closer to $m_i$ than to $m_{3-i}$ for $i\in\{1,2\}$. 
 Unfortunately these constraints do not entirely characterize these objects (see next section) but they do, in some sense, asymptotically. Thinking heuristically for a moment, we can expect that the analogue in the continuum limit of these discrete configurations is a Brownian surface with four marked points $(m_1^\infty, m_2^\infty, s_1^\infty, s_2^\infty)$ such that if we subdivide the space in two nearest-neighbour cells induced by  $m_1^\infty$ and $m_2^\infty$, the point $s_i^\infty$ belongs to the nearest-neighbour cell induced by $m_i^\infty$ for each $i\in\{1,2\}$.
Up to technical details that we will carry out in the next section, this leads us quite naturally to the following conclusion:
\begin{lemma}\label{lemma:convergenceCell}
The coefficient $[z^n]A_g(z)$ is such that, as $n$ goes to $\infty$:
\begin{eqnarray}\label{eq:convergenceCell}
\frac{[z^n]A_g(z)}{3/2 \cdot n^3m_{g}(n)} \longrightarrow \mathbf{E}[\XX_g(1-\XX_g)]
\end{eqnarray}
with the notation of Theorem~\ref{thm:obs1}.
\end{lemma}
\begin{rem}\label{rem:1} The reader can understand heuristically the meaning of the denominator $3/2 \cdot n^3 m_{g}(n)$ as follows. The tuple $(Q,s_1,s_2,e_1,e_2)$ is a quadrangulation with two marked vertices and two marked edges. We can use $e_1$ as the root-edge of $Q$, and orient it by deciding that its source is at even distance from~$s_1$. We can choose the ``error'' $\epsilon$ freely in $\{-1,0,1\}$ (since asymptotically we do not expect this error to play any role), and set $i_1:=d(s_1,m_1)$ and $\delta:=i_1+\epsilon - d(s_2,m_2)$. Since Miermont's bijection requires that $d(s_1,s_2)+\delta$ is even, we are left with a rooted quadrangulation with one marked edge~$e_2$, and two marked vertices $(s_1,s_2)$ subject to \emph{two} parity constraints (that $s_1$ is at even distance from the root, and that $d(s_1,s_2)+\delta$ is even). Since a quadrangulation with $n$ faces has $2n$ edges and $n+2-2g$ vertices, and since it is natural to expect each parity constraint to contribute an asymptotic factor $\frac{1}{2}$, the total number of ``base configurations'' we obtain is $\sim 3\times (2n) n^2/4 \cdot m_g(n)$, hence the denominator in~\eqref{eq:convergenceCell}. %
\end{rem}
\medskip

Admitting Lemma~\ref{lemma:convergenceCell} (to be proved in Section~\ref{sec:technical}) we can now conclude the proof of Theorem~\ref{thm:obs1}. First, we can rewrite the decomposition equation~\eqref{eq:Tutte} as:
\begin{eqnarray}\label{eq:TutteKernel}
(1-6zL_0(z)) L_{g+1}(z) 
- 
3z\sum_{g_1+g_2=g+1,\atop g_1,g_2>0}
L_{g_1}(z) 
L_{g_2}(z) 
=
z A_g(z),
\end{eqnarray}
which expresses the generating function $L_{g+1}(z)$ in terms of the lower genus functions $L_i(z)$ for $i\leq g$, and of the ``unknown'' quantity $A_g(z)$.
We recall that $Q^\bullet_g(z)=2L_g(z)$ and~\eqref{eq:singQg}, from which we observe that each term in the L.H.S. of~\eqref{eq:TutteKernel} has a dominant singularity at $z=\tfrac{1}{12}$ with the same order of magnitude. More precisely, for the first term, using~\eqref{eq:valKernel}, we obtain $(1-6zL_0(z))L_{g+1}(z)\sim 2^{1-5(g+1)}\tau_{g+1}(1-12z)^{1-\frac{5}{2}(g+1)}$.
For product terms we have 
$L_{g_1}(z) L_{g_2}(z)\sim
2^{2-5(g+1)}\tau_{g_1}\tau_{g_2}(1-12z)^{1-\frac{5}{2}(g_1+g_2)}$.
It follows, using standard transfer theorems for algebraic functions~\cite{Flajolet} that when $n$ goes to infinity:
\begin{eqnarray}
[z^{n-1}] A_g(z) \sim 12^n n^{\frac{5g+1}{2}} 
2^{1-5(g+1)}\Gamma\left(\tfrac{5g+3}{2}\right)^{-1}
\left(\tau_{g+1} - \tfrac{1}{2}\sum_{g_1+g_2=g+1\atop g_1,g_2>0} \tau_{g_1}\tau_{g_2} \right).
\end{eqnarray}
But from Lemma~\ref{lemma:convergenceCell}, we have another expansion of the ``unknown'' coefficient $[z^{n-1}] A_g(z)$, namely:
$$
[z^{n-1}] A_g(z) \sim\mathbf{E}\XX_g(1-\XX_g) \cdot 3/2\cdot n^3 m_{g}(n-1) 
\sim\mathbf{E}\XX_g(1-\XX_g)\cdot  3\cdot 2^{1-5g}\Gamma(\tfrac{5g-1}{2})^{-1} \tau_g n^{\frac{5g+1}{2}} 12^{n-1}.
$$
Theorem~\ref{thm:obs1} follows by comparing the last two expansions of the ``unknown'' quantity $[z^{n-1}]A_g(z)$ (we recall that $\Gamma(\frac{5g+3}{2})/\Gamma(\tfrac{5g-1}{2})=\frac{(5g+1)(5g-1)}{4}$).

\subsection{Remaining proofs, I: general properties}
Because we will need to pick both edges and vertices at random, we first need a lemma that compares both:
\begin{lemma}\label{lemma:vertexvsedge}
Given a quadrangulation $Q$ of genus $g$ with $n$ faces, there exists a probability measure $\mu_n^E$ on edges of $Q$ and a mapping $\phi:E(Q)\rightarrow V(Q)$ that associates to each edge of $E$ a vertex at distance at most one of one of its endpoints, such that the distance in total variation between $\mu_n^{E}$ and the uniform measure on edges, and between $\phi \circ \mu_n^{E}$ and the uniform measure on vertices, are both $O(\tfrac{1}{n})$.
\end{lemma}
\begin{proof}
Let $L$ be the l.1.f.m associated to $Q$ via the Marcus-Schaeffer bijection and let $v_0\in V(L)$. The map $L$ has $n+1-2g$ vertices, so it is possible to choose a set $E'_L$ of $n-2g$ edges of $L$ and an orientation of edges of $E'_L$ such that each vertex of $V\setminus\{v_0\}$ has exactly one outgoing edge from $E'_L$ (to see this, take a spanning tree of $L$ and orient edges towards $v_0$). If $e \in E'_L$, we let $v(e)$ be its source, which is an element of $V\setminus\{v_0\}$. The edge $e$ of $E'_L$ is associated, via the Marcus-Schaeffer bijection, to a face $f(e)$ of the quadrangulation $Q$ that is incident to the vertex $v(e)$. We let $E'_Q$ be the subset of edges of $Q$ that border a face of the form $f(e)$ for some $e\in E_L$ and that are oriented from white to black when going clockwise around $f(e)$ (in some fixed bicoloration of $Q$). If $\tilde{e}\in E'_Q$, corresponding to the face $f(e)$, we define $\phi(\tilde{e}):=v(e)$. Since $\tilde{e}$ and $\psi(\tilde{e})$ both border the face $f(e)$, they are at distance at most one from each other. Moreover, if we choose $\tilde{e}$ uniformly at random from $E'_Q$, then by construction $\psi(\tilde{e})$ is uniform in $V\setminus \{v_0\}$. 
	Since $E'_Q$ contains $2(n-2g)$ edges of $Q$ (among $2n$) and $V\setminus\{v_0\}$ contains $n-2g$ vertices of $Q$ (among $n+2-2g$), we can choose $\phi(e)$ arbitrarily among endpoints of $e$ for $e \not \in E'_Q$, and we are done.
\end{proof}

In the following discussion we will implicitly restrict ourselves to a subsequence along which we have the GHP distributional convergence:
$$
(\qgn,\tfrac{1}{n^{1/4}}\mathbf{d}_n,\mu_n)
\longrightarrow
(\qginf, d_\infty, \mu_\infty).
$$
We will need the following direct consequence of~\cite[Thm. 4]{Chapuy:trisections}. We state separately a discrete and a continuous statement, although they are intimately related: 
\begin{lemma}\label{lemma:zeroProba} ~~

(i) Let $(\qginf,d_\infty, \mu_\infty)$ be a Brownian surface of genus $g$ and let $(\vv_1^\infty, \vv_2^\infty, \vv_3^\infty)$ chosen at random according to $\mu_\infty^{\otimes 3}$. Then almost surely we have 
$d_\infty (\vv_1^\infty,\vv_3^\infty) \neq 
d_\infty (\vv_2^\infty,\vv_3^\infty)$. 

(ii) Fix $K\geq 0$, and pick three uniform random vertices $\vv_1^n$, $\vv_2^n$, $\vv_3^n$ in $\qgn$. Then the probability that $|d(\vv_1^n,\vv_3^n)-d(\vv_2^n,\vv_3^n)|\leq K$ goes to zero when $n$ goes to infinity.
\end{lemma}
\begin{proof}
It is proved in~\cite{Chapuy:trisections} that if $\vv_1\in_u V(\qgn)$, the random measure 
$$\eta_n^{(g)}:=\displaystyle \tfrac{1}{|V(\qgn)|}\sum_{v \in V(\qgn)} \delta_{d_n(v,\vv_1)/n^{1/4}}$$ 
converges in distribution to a random measure $\eta_g$ that, almost surely, has no atoms (the latter fact following from the fact that it is true for the ISE measure, see e.g.~\cite{MBMJanson}, and from the relation between $\eta_g$ and ISE given in~\cite{Chapuy:trisections}).
Now for $\alpha>0$, let $\pnalpha:=\mathbf{P}\big\{|d(\vv_1^n,\vv_3^n)-d(\vv_2^n,\vv_3^n)|\leq \alpha n^{1/4}\big\} = \mathbf{E}\langle ({\eta_n^{(g)}})^{\otimes 2} | h_\alpha\rangle$ where $h_\alpha(x,y):=\mathbf{1}_{|y-x|\leq \alpha}$ (here for a measure $\nu$ and a function $h$  we note $\langle\nu|h\rangle:=\int h(x)d\nu(x)$).
Convergence in law implies that $\lim_n \mathbf{E}\langle(\eta_n^{(g)})^{\otimes 2} | f \rangle = \mathbf{E} \langle \eta_g^{\otimes 2}| f\rangle$ for any bounded and continuous function $f$, so choosing $f=f_\alpha$ continuous such that $h_\alpha \leq f_\alpha \leq h_{2\alpha}$ we get 
$\limsup_n p^n_\alpha \leq \mathbf{E}\langle ({\eta_g})^{\otimes 2} | h_{2\alpha}\rangle$, and since $\eta_g$ has no atoms we get:
\begin{eqnarray}\label{eq:limsup}
\lim_{\alpha\rightarrow 0} \limsup_n \pnalpha =0,
\end{eqnarray}
from which (ii) follows (in fact, in a much stronger form that allows $K$ to be as large as $o(n^{1/4})$).

Now, by GHP convergence and \cite[Prop. 6]{Miermont:tessellations} one can define on the same probability space $(\qgn,\tilde\vv_1^n,\tilde\vv_2^n,\tilde\vv_3^n)$ and $(\qginf,\tilde\vv_1^\infty, \tilde\vv_2^\infty, \tilde\vv_3^\infty)$ such that $\tilde\vv_i^{M}$ is $o(1)$-close to $\vv_i^{M}$ in total variation distance for each $i\in\{1,2,3\}$ and $M\in\{n,\infty\}$, and such that almost surely $|d_n(\tilde\vv_i^n,\tilde\vv_j^n)/n^{1/4}-d_\infty(\tilde\vv_i^\infty, \tilde\vv_j^\infty)|=o(1)$ for each $i$, $j$.
Letting $q_\alpha:=\mathbf{P}\{|d_\infty (\vv_1^\infty,\vv_3^\infty) - 
d_\infty (\vv_2^\infty,\vv_3^\infty)|\leq \alpha\}$, it easily follows that 
$$q_{\alpha/2} \leq \limsup_n \pnalpha.$$
From \eqref{eq:limsup} this implies that $\limsup_{\alpha\rightarrow 0} q_\alpha =0$, which implies (i). 
\end{proof}

\subsection{Remaining proofs, II: Lemma~\ref{lemma:convergenceCell}}\label{sec:technical}

Before proving Lemma~\ref{lemma:convergenceCell}, we need to describe more precisely the objects Miermont's bijection leaves us with.
We use the same notation as in the previous section for objects counted by $A_g(z)$ (marked faces $F_1, F_2$, minimum label in each face $0,\delta$, marked corners $c_1,c_2$). We will introduce the refinement
$$
A_g(z) = \sum_{\epsilon\in\{-1,0,1\}} A_g^\epsilon(z)
$$
where $A_g^\epsilon(z)$ counts the same objects as $A_g(z)$ but with the restriction that $\ell(c_2)-\ell(c_1)=\epsilon$.
Then we have:
\begin{lemma}%
	\label{lemma:MiermontConstraints}
For each $\epsilon\in\{-1,0,1\}$, the labelled two-face maps counted by $[z^{n}]A^{\epsilon}_g(z)$ are in bijection with tuples $(Q,s_1,s_2,e_1,e_2)$ such that:
\begin{itemize}[itemsep=0pt,topsep=0pt, parsep=0pt, leftmargin=30pt] 
\item[(M0)] $Q$ is a bipartite quadrangulation of genus $g$ with $n$ faces (unrooted);
\item[(M1)] $s_1,s_2$ are two vertices of $Q$ and $e_1,e_2$ are two marked edges of $Q$;
\item[(M2)] for $j\in\{1,2\}$ let $m_j$ be the endpoint of $e_j$ closer from $s_j$, %
and let $\delta := d(s_1,m_1)+\epsilon-d(s_2,m_2).$
Then the quantity $d(s_1,s_2)+\delta$ is even.
\item[(M3)]
	Label each vertex $v$ of the quadrangulation $Q$ by $\ell(v):=\min(d(v,s_1),d(v,s_2)+\delta)$, orient each edge towards its vertex of minimum label, and define leftmost geodesic paths as in Section~\ref{sec:prel1}.
		Then for $i\in\{1,2\}$, the leftmost geodesic path starting at $e_i$ ends at $s_i$.
\end{itemize}
\end{lemma} 
\begin{proof} This follows from Miermont's bijection \cite[Thm. 4]{Miermont:tessellations} as described in Section~\ref{sec:prel1}, applied to the $2$-pointed quadrangulation $(Q;s_1,s_2;0,\delta)$. The only subtle point is to notice that condition~\eqref{eq:miermont2} is ensured by (M3), for if \eqref{eq:miermont2} did not hold, either $s_1$ or $s_2$ would \emph{not} be a local minimum of the labelling and the leftmost geodesic paths would not stop at that vertex\footnote{We thank a referee for this remark.}.
\end{proof}

Note that, since $Q$ is bipartite, the property $(M2)$ is equivalent to the following:
\begin{itemize}[itemsep=0pt,topsep=0pt, parsep=0pt, leftmargin=30pt] 
\item[(M'2)]
$d(m_1,m_2) \equiv \epsilon \mod 2.$
\end{itemize}
As for the complicated property $(M3)$, up to subdominating cases, it can be rephrased in simpler terms closely related to nearest neighbours tessellations. Indeed, we have:
\begin{lemma}\label{lemma:GeoToVor}
Let $b_g^\epsilon(n)$ be the number of tuples $(Q,s_1,s_2,e_1,e_2)$ satisfying  (M0), (M1), (M2) of the last lemma, and such that moreover we have:
\begin{itemize}[itemsep=0pt,topsep=0pt, parsep=0pt, leftmargin=30pt] 
\item[(M'3)]
$d(s_1,e_1) < d(s_1,e_2) -4$ and $ 
d(s_2,e_2) < d(s_2,e_1) -4 .$
\end{itemize}
Then for each $\epsilon\in \{-1,0,1\}$ we have $[z^n]A_g^{\epsilon}(z) = b^\epsilon_g(n) + o(n^3m_g(n))$.
\end{lemma}
Note that $(M'3)$ implies  that $d(s_1,m_1) < d(s_1,m_2) -2$ and $ 
d(s_2,m_2) < d(s_2,m_1) -2.$

\begin{proof}
Let $(Q,s_1,s_2,e_1,e_2)$ satisfying the hypotheses of Lemma~\ref{lemma:GeoToVor}. Then we claim that it also satisfies the hypotheses of Lemma~\ref{lemma:MiermontConstraints}. Indeed define $\ell(v)=\min(d(s_1,v),d(s_2,v)+\delta)$ as in Lemma~\ref{lemma:MiermontConstraints}. We observe that (M'3) implies that 
$$ d(s_1,m_1) < d(s_1,m_1)+(d(s_2,m_1)-d(s_2,m_2)+\epsilon) = d(s_2,m_1)+\delta,$$
which shows that the minimum in the definition of $\ell(v)$ for $v=m_1$ is reached only by its \emph{first} argument. Similarly, for $v=m_2$ we have by (M'3) that:
$$ d(s_2,m_2) < d(s_2,m_2)+(d(s_1,m_2)-d(s_1,m_1)-\epsilon)
=d(s_1,m_2)-\delta,
$$
which shows that the minimum in the definition of $\ell(v)$ for $v=m_2$ is reached only by its \emph{second} argument.
	Therefore, the first inclusion in~\eqref{eq:inclusionCells}
	precisely says that (M3) is satisfied.

\smallskip

Conversely assume the hypotheses of Lemma~\ref{lemma:MiermontConstraints}. By the second inclusion in~\eqref{eq:inclusionCells},
property $(M3)$ ensures that
the minima defining $\ell(m_1)=\min(d(m_1,s_1),d(m_1,s_2)+\delta)$ and $\ell(m_2)=\min(d(m_2,s_1),d(m_2,s_2)+\delta)=\min(d(m_2,s_1),d(m_1,s_1)+\epsilon)$ are reached respectively by their first and second argument (and possibly reached twice). This implies:
$$
d(s_1,m_1)+\epsilon \leq d(s_1,m_2), \ d(s_2,m_2)\leq d(m_1,s_2)+\epsilon.
$$
Thus, if hypothesis (M'3) is \emph{not} satisfied, it must hold that either $|d(s_1,m_1)-d(s_1,m_2)|\leq 2$ or $|d(s_2,m_2)-d(s_2,m_1)|\leq 2$. It thus suffices to show that there are at most $o(n^3m_g(n))$ tuples $(Q,s_1,s_2,e_1,e_2)$ such that one of these two properties holds. For this it suffices to show that if $(\qq,\ee_1) \in_u \QQgn$ is a random rooted quadrangulation and $(\sst_1,\sst_2,\ee_2)$ are two vertices and an edge chosen independently uniformly at random in $\qq$, the probability that $|d(\sst_1,\ee_2)-d(\sst_1,\ee_2)|\leq 2$ or $|d(\sst_2,\ee_2)-d(\sst_2,\ee_1)|\leq 2$ goes to zero as $n$ goes to infinity. This directly follows from Lemmas~\ref{lemma:vertexvsedge} and~\ref{lemma:zeroProba}(ii).
\end{proof}

In view of getting rid of the constraint $(M'2)$, we state the following lemma:
\begin{lemma}\label{lemma:equalContrib}
For any $\epsilon\in \{-1,0,1\}$ we have as $n$ goes to infinity:
$$
[z^n] A^\epsilon_g  (z) \sim \frac{1}{3} [z^n]A_g(z).
$$
\end{lemma} 
\begin{proof}[Proof (sketch)]
This can be proved by asymptotic analysis of generating functions using a simple adaptation of the method developed in \cite{CMS} for the enumeration of labelled one-face maps by \emph{scheme decomposition}: one can enumerate objects counted by $A^\epsilon_g(z)$ with this approach and realize that changing the parameter $\epsilon$ only affects the main term of the singular expansion by a factor $1-O((1-12z)^{1/4})$, from which the result follows.
We leave details to the reader.
\end{proof}

We are now ready to conclude the proof. 
\newcommand{\tqgn}{{\tilde{\mathbf{q}}}^{(g)}_n}
\begin{proof}[{Proof of Lemma~\ref{lemma:convergenceCell}}]

First, we remark that from the last lemma:
\begin{eqnarray}\label{eq:linearSmart}
[z^n] \big(A_g^0(z) + A_g^1(z)\big) \sim \frac{2}{3} [z^n] A_g(z),
\end{eqnarray}
while from Lemma~\ref{lemma:GeoToVor} and the remark preceeding it, $[z^n] A_g^0(z) + A_g^1(z)$ is equivalent to the number of tuples $(Q,s_1,s_2,e_1,e_2)$ satisfying properties $(M0),(M1)$, and $(M'3)$ (note that property $(M'2)$ disappears since we sum over both parities $\epsilon=0,1$). We will thus focus on such objects in the rest of the proof.

We note that a bipartite quadrangulation with a marked edge $e_1$ and a marked vertex $v_1$ can be canonically rooted by orienting $e_1$ towards its unique endpoint at even distance from $v_1$. This gives a one-to-two correspondence between elements of $\mathcal{Q}^{g}_n$ with a marked vertex and (unrooted) bipartite quadrangulations of genus $g$ with a marked vertex and a marked unoriented edge.
We thus have:
\begin{eqnarray}\label{eq:proba}
\frac{[z^n]\big(A_g^0(z) + A_g^1(z)\big)}
{(2n) \cdot n^2 \cdot  m_g(n)}
\sim \frac{1}{2}\cdot \mathbf{P}\Big\{d(\sst_1,\ee_1) < d(\sst_1,\ee_2) -4, \  
d(\sst_2,\ee_2) < d(\sst_2,\ee_1) -4 \Big\},
\end{eqnarray}
 where the probability is taken over $\qgn\in_u \mathcal{Q}^g_n$ with two uniform marked vertices $\sst_1,\sst_2$, a uniform marked edge $\ee_2$ and $\ee_1$ is the root edge (in the denominator, the factor $n^2$ corresponds to the choice of the two vertices, while the factor $(2n)$ corresponds to the choice of the edge $\ee_2$).

We recall that we implicitly restrict ourselves to a subsequence along which  the GHP distributional convergence
$$
(\qgn,\tfrac{1}{n^{1/4}}\mathbf{d}_n,\mu_n)
\longrightarrow
(\qginf, d_\infty, \mu_\infty)
$$
holds. We will make use of this convergence using a coupling between $\qgn$ and $\qginf$.
 More precisely, according to \cite[Proposition~6]{Miermont:tessellations}, we can build $\qgn$ and $\qginf$ on the same probability space, and define a measure $\nu$ on $\qgn \times \qginf$ such that for each $k\geq 1$, if $(\tilde\ww^i_n, \tilde\ww^i_\infty)_{1\leq i \leq k} \sim \nu^{\otimes k}$ we have almost surely $|d_n(\tilde\ww^i_n,\tilde\ww^j_n)n^{-1/4}-d_\infty(\tilde\ww_\infty^i, \tilde\ww_\infty^j)| \leq \epsilon_n$ for any $i,j$, and moreover the law of $\tilde\ww^i_n$ (resp. $\tilde\ww^i_\infty)$ differs from $\mu_n$ (resp. $\mu_\infty$) by at most $\epsilon_n$ in total variation distance, where $\epsilon_n$ is a nonnegative real sequence going to zero when $n$ goes to infinity.
We will apply this with $k=4$. Using Lemma~\ref{lemma:vertexvsedge}, we can moreover assume the vertices $\tilde\ww_3^n$ and $\tilde\ww_4^n$ are at distance at most $2$ of two random edges $\tilde\ee_1^n$ and $\tilde\ee_2^n$ respectively, and that the law of $\tilde\ee_1^n$ and $\tilde\ee_2^n$ is $\epsilon_n$-close in total variation to that of two uniform random edges (if necessary, we modify the sequence $\epsilon_n$ for this to be true, still asking that $\epsilon_n\rightarrow 0$). 

If $v_1,v_2,v_3,v_4$ are points (or subsets) in some metric space of underlying distance~$d$, and $K\in\mathbb{R}$ let us define the events:
$$V_K(v_1,v_2,v_3,v_4) := \{ d(v_1,v_3)<d(v_1,v_4)-K, d(v_2,v_4)<d(v_2,v_3)-K\}
$$
$$
W_K(v_1,v_2,v_3):=\{|d(v_1,v_3)-d(v_2,v_3)|\leq K\}.
$$
By the assumptions made on the coupling between $\qgn$ and $\qginf$ and from the triangle inequality we have, denoting $\Delta$ the symmetric difference:
\begin{align*}
V_{4}(\tilde\ww_1^n, \tilde\ww_2^n, \tilde\ee_1^n,\tilde\ee_2^n)
\Delta
V_0^\infty(\tilde\ww_1^\infty, \tilde\ww_2^\infty, \tilde\ww_3^\infty,\tilde\ww_4^\infty)
\subset W_{\delta_n}^\infty(\tilde\ww_1^\infty, \tilde\ww_2^\infty, \tilde\ww_3^\infty)
\cup W_{\delta_n}^\infty(\tilde\ww_1^\infty, \tilde\ww_2^\infty, \tilde\ww_4^\infty)
\end{align*}
with $\delta_n = O(\epsilon_n+n^{-1/4})$.
We thus have:
\begin{align*}
\limsup_n \big|\mathbf{P}
V_{4}(\tilde\ww_1^n, \tilde\ww_2^n, \tilde\ee_1^n,\tilde\ee_2^n)
-\mathbf{P}
V_0(\tilde\ww_1^\infty, \tilde\ww_2^\infty, \tilde\ww_3^\infty,\tilde\ww_4^\infty)
\big|
&\leq
\limsup_n
2 \mathbf{P}W_{\delta_n}^\infty(\tilde\ww_1^\infty, \tilde\ww_2^\infty, \tilde\ww_3^\infty) \\
&\leq \limsup_n \left(\epsilon_n+
2 \mathbf{P}W_{\delta_n}^\infty(\ww_1^\infty, \ww_2^\infty, \ww_3^\infty)\right)
\end{align*}
where $(\ww_i^\infty)_{1\leq i\leq 3} \sim \mu_\infty^{\otimes 3}$ are three  random vertices in $\qginf$  chosen according to $\mu_\infty$, and where we just used the definition of total variation distance. From Lemma~\ref{lemma:zeroProba}(i), the last $\limsup$ is equal to zero, which implies:
\begin{align*}
\limsup_n \big|\mathbf{P}
V_{4}(\ww_1^n, \ww_2^n, \ee_1^n,\ee_2^n)
-\mathbf{P}
V_0(\ww_1^\infty, \ww_2^\infty, \ww_3^\infty,\ww_4^\infty)
\big|
&=0
\end{align*}
where $\ww_1^n, \ww_2^n, \ee_1^n,\ee_2^n$ are two vertices and two edges of $\qgn$ chosen independently uniformly at random, and where $(\ww_i^\infty)_{1\leq i\leq 4} \sim \mu_\infty^{\otimes 4}$ are uniform in $\qginf$.

Now by rerooting invariance of random quadrangulations $\mathbf{P}
V_{4}(\ww_1^n, \ww_2^n, \ee_1^n,\ee_2^n)$ is equal to the probability appearing in the R.H.S. of~\eqref{eq:proba}, while it follows directly from Lemma~\ref{lemma:zeroProba}(i) and the Fubini theorem that the quantity $
\mathbf{P} 
V_0(\ww_1^\infty, \ww_2^\infty, \ww_3^\infty,\ww_4^\infty)$ is equal to
$\mathbf{E}\XX_g(1-\XX_g)$. This concludes the proof.
\end{proof}

To be fully complete we also state the:
\begin{proof}[Proof of Theorem~\ref{cor:main}]
The only thing to prove is that $\mathbf{E}\mathbf{X}_g = \tfrac{1}{2}$, which is a direct consequence of Lemma~\ref{lemma:zeroProba}.
\end{proof}

\section{Three marked points (proof of Theorem~\ref{thm:3points})}

In this section we sketch the proof of Theorem~\ref{thm:3points}. We will insist on the combinatorial decompositions and the computation, since the details of the convergence results are very similar to what we did in the previous section.

We first need some definitions from~\cite{CMS, Chapuy:trisections}. If $L$ is a one-face map, its \emph{skeleton} is the map obtained by removing all vertices of degree~$1$ in $L$, and continuing to do so recursively until only vertices of degree at least $2$ remain. 
Vertices of a one-face map that are vertices of degree at least $3$ of its skeleton are called \emph{nodes}. A node $v$ that has degree $k$ in the skeleton is called a $k$-node (note that its degree as a vertex in the one-face map can be larger than $k$). A one-face map is \emph{dominant} if all vertices of its skeleton have degree at most $3$, \textit{i.e.} if all its nodes are $3$-nodes. It is proved in \cite{CMS} that for fixed $g$, as $n$ goes to infinity, a proportion at least $1-O(n^{-1/4})$ of l.1.f.m. of genus $g$ with $n$ edges are dominant.
By Euler's formula, a dominant one-face map has $4g-2$ nodes.

Following~\cite{Chapuy:trisections}, we introduce the operation of \emph{opening}. If $L$ is a one-face map and $v$ is a $3$-node of $L$, the \emph{opening} of $v$ is the operation that consists in replacing $v$ by three new vertices, each linked to one edge of the skeleton, and distributing the three (possibly empty) subtrees attached to $v$ among these new vertices as on the following figure:
\begin{center}\includegraphics[width=0.4\linewidth]{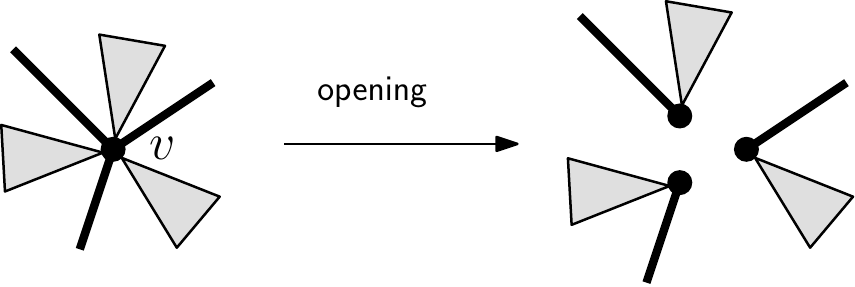}\end{center}
Following~\cite{Chapuy:trisections}\footnote{Strictly speaking, these notions are defined only for dominant maps in that reference. Here it will be convenient for the presentation of the decompositions to extend them to general $3$-nodes -- but this is not a fundamental need, since all quantities involved in our discussion will be led by dominant maps at the first order.}, we distinguish two types of $3$-nodes in a one-face map: \emph{intertwined nodes}, that are such that their opening results in a one-face map of genus $g-1$ with three marked vertices; and \emph{non-intertwined nodes}, that are such that their opening results in one or more maps of total genus $g-2$ with three faces in total, and one marked vertex inside each face. The \emph{trisection lemma}~\cite[Lemma~5]{Chapuy:trisections}, which is the key result underlying this section, asserts that any dominant map of genus $g\geq 1$ has exactly $2g$ intertwined nodes, hence $2g-2$ non-intertwined ones.

It follows that the number $K_{g+2}(n)$ of l.1.f.m. of genus $g+2$ with $n$ edges whose root edge is a skeleton-edge leaving a \emph{non-intertwined} $3$-node satisfies:
\begin{eqnarray}\label{eq:doubleRooting}
K_{g+2}(n)\sim\frac{6(g+1)}{2n} [z^n] L_{g+2}(z).
\end{eqnarray}
Indeed, the first-order contribution is given by dominant l.1.f.m., and in a dominant l.1.f.m. of genus $g+2$ we can choose $3(2(g+2)-2)=6(g+1)$ edges outgoing from a non-intertwined node as a new root edge, but we obtain each map $2n$ times in this way (since maps counted by $L_{g+2}(z)$ are already rooted at one of their $2n$ oriented edges). 

\medskip

We are now going to obtain another expression for the number $K_{g+2}(n)$ by performing a combinatorial decomposition. Comparing the two expressions will, in the end, lead us to Theorem~\ref{thm:3points}.

Let $L$ be a dominant l.1.f.m of genus $g+2$ whose root edge is a skeleton-edge leaving a non-intertwined $3$-node $v$. We distinguish three cases, according to what happens when we perform the opening of the node $v$ (see Figure~\ref{fig:Tutte3}):
\begin{itemize}[itemsep=0pt, topsep=0pt,parsep=0pt, leftmargin=20pt]
\item[(i)] we disconnect the map into three components;
\item[(ii)] we disconnect the map into two components; 
\item[(iii)] we do not disconnect the map.
\end{itemize}
We let $C_{g+2}^{(i)}(z)$, $C_{g+2}^{(ii)}(z)$, $C_{g+2}^{(iii)}(z)$ be the generating function for these three cases, respectively.

\begin{figure}
\begin{center}
\includegraphics[width=\linewidth]{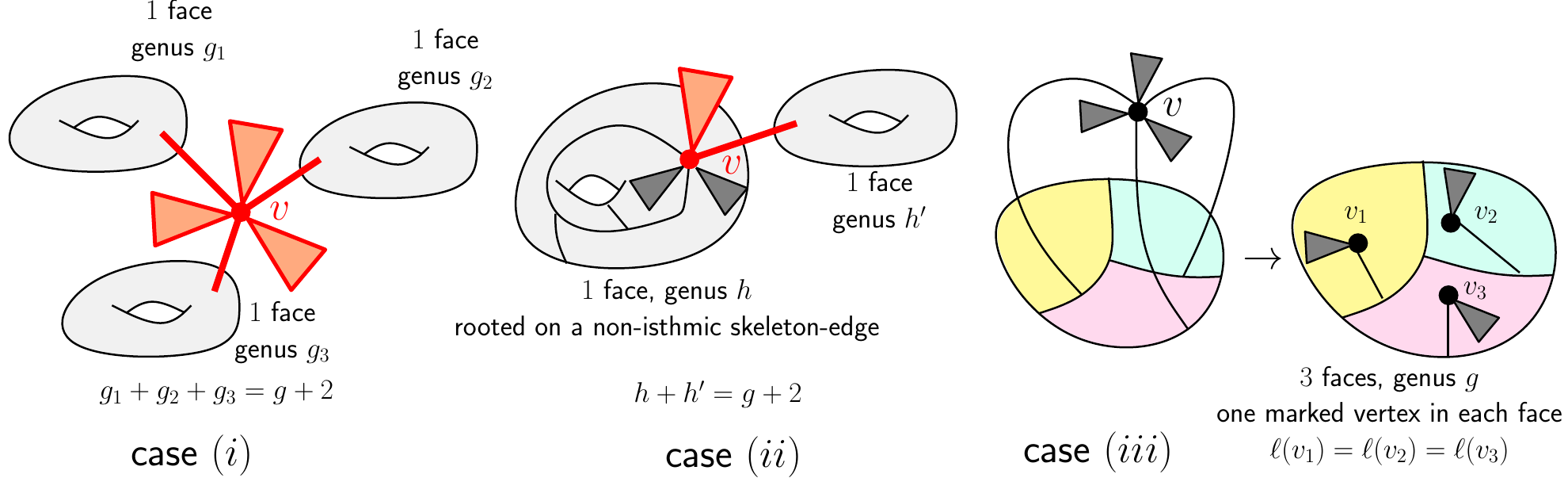}
\caption{The three cases for a one-face map of genus $g+2$ rooted at a skeleton edge leaving a non-intertwined $3$-node $v$.}\label{fig:Tutte3}
\end{center}
\end{figure}

Configurations corresponding to (i) can be reconstructed by starting with three rooted l.1.f.m. of positive genera summing up to $g+2$, and joining the three root vertices by new edges to a new vertex $v$. The generating function for the contribution of this case is thus:
\begin{eqnarray}\label{eq:case1}
C_{g+2}^{(i)}(z) = (3zL_0(z))^3
\sum_{g_1+g_2+g_3=g+2\atop g_1,g_2,g_3>0} L_{g_1}(z)L_{g_2}(z)L_{g_3}(z)
\end{eqnarray}
where for each new edge a factor $3zL_0(z)$ takes into account the increment of this edge, and the attachment of a rooted tree (possibly empty) in the newly created corner (see Figure~\ref{fig:Tutte3}--Left). 

Configurations corresponding to (ii) can be reconstructed by joining with a new edge the root vertex of  a l.1.f.m. to the root vertex of another one which is rooted at a \emph{non-isthmic} edge of its skeleton, and possibly adding a tree in the newly created corner (see Figure~\ref{fig:Tutte3}--Center; the  genera of these two maps sum to $g+2$). Now, arguing as in the previous section, for each $h\geq 1$, the generating function $S_{h}(z)$ of l.1.f.m. of genus $h$ rooted at a non-isthmic edge of their skeleton satisfies:
$$
L_h(z) = S_h(z) + 3z\sum_{g_1+g_2=h\atop g_1,g_2\geq 0} L_{g_1}(z)L_{g_2}(z),
$$
from which we get:
$$
S_h(z)=(1-6zL_0(z))L_h - 3z \sum_{g_1+g_2=h\atop g_1,g_2>0} L_{g_1}(z)L_{g_2}(z).$$
It follows that the contribution for case (ii) is given by:
\begin{eqnarray}\label{eq:case2}
C_{g+2}^{(ii)} = 
3\cdot (3z L_0(z)) \sum_{h+h'=g+2\atop h,h'>0} L_{h'}(z) \left(
(1-6zL_0(z))L_h - 3z \sum_{g_1+g_2=h\atop g_1,g_2>0} L_{g_1}(z)L_{g_2}(z)\right),
\end{eqnarray}
where as before the factor $3zL_0(z)$  takes into account the increment of the newly created edge, and the (possibly empty) rooted tree to attach in the newly created corner, and where the global factor of $3$ takes into account the choice of the root edge among the three skeleton-edges incident to the newly created vertex.

Summing up~\eqref{eq:case1} and~\eqref{eq:case2} we obtain the leading-order contribution for the sum of the first two cases:
$$
C_{g+2}^{(i)}(z) + C_{g+2}^{(ii)} (z) \sim C\cdot (1-12z)^{\frac{3}{2}-\frac{5}{2}(g+2)}
$$
with
\begin{eqnarray*}
C&=& 2^{-5(g+2)} \sum_{\scriptscriptstyle g_1+g_2+g_3=g+2} \tau_{g_1}\tau_{g_2}\tau_{g_3}
+
\frac{3\cdot 4}{2} 2^{-5(g+2)} 
\sum_{\scriptscriptstyle g_1+g_2=g+2} \tau_{g_1}\tau_{g_2} 
-
3\cdot 2^{-5(g+2)} 
\sum_{\scriptscriptstyle g_1+g_2+g_3=g+2} \tau_{g_1}\tau_{g_2} \tau_{g_3}\\
&=&
2^{-5(g+2)} \left(
6\cdot 
\sum_{\scriptscriptstyle g_1+g_2=g+2} \tau_{g_1}\tau_{g_2} 
-2\cdot
\sum_{\scriptscriptstyle g_1+g_2+g_3=g+2} \tau_{g_1}\tau_{g_2} \tau_{g_3}
\right)
\end{eqnarray*}
where we have used that $L_0(\frac{1}{12})=2$, and all sums are taken over positive indices ({\it i.e.} $g_1,g_2,g_3 >0$).

\smallskip

Now, the leading-order contribution for the sum of the three cases (i), (ii), (iii), which from \eqref{eq:doubleRooting} corresponds to the dominant singularity of the generating function 
$3(g+1) \int L_{g+2}(z) \frac{dz}{z}$ is given by (since all series are algebraic we can integrate expansions with no fear):
$$
3(g+1) \frac{1}{\frac{3}{2}-\frac{5}{2}(g+2)} 2^{1-5(g+2)} \tau_{g+2} (1-12z)^{\frac{3}{2}-\frac{5}{2}(g+2)}.
$$
Taking the difference with the previous expression, we obtain that the leading order contribution corresponding to case (iii) is given by 
\begin{eqnarray}\label{eq:case3}
C_{g+2}^{(iii)}(z) \sim C' 2^{-5(g+2)}(1-12)^{\frac{3}{2}-\frac{5}{2}(g+2)}
\end{eqnarray} where:
$$
C'= \frac{12(g+1)}{5g+7} \tau_{g+2}-\left(
6\cdot 
\sum_{\scriptscriptstyle g_1+g_2=g+2} \tau_{g_1}\tau_{g_2} 
-2\cdot
\sum_{\scriptscriptstyle g_1+g_2+g_3=g+2} \tau_{g_1}\tau_{g_2} \tau_{g_3}
\right) .
$$
This expression can be considerably simplified. To this end, define the formal power series $U(s)=\sum_{g\geq 1} \tau_g s^g$. Then the $t_g$-recurrence is equivalent to the equation:
\begin{eqnarray}\label{eq:tgdiff}
U(s)=\tfrac{s}{3} + \tfrac{1}{2}U(s)^2 + \frac{s}{3} 
(5(\tfrac{sd}{ds})+1)
(5(\tfrac{sd}{ds})-1)
U(s).
\end{eqnarray}
In view of the bi- and tri-linear sums appearing in the definition of $C'$, we would like to find an equation involving the series $6U(s)^2-2U(s)^3$.
Luckily, we have:
\begin{lemma}\label{lemma:eliminate}
The following differential equation holds:
\begin{eqnarray}\label{eq:tgdiff2}
\tfrac{4}{15} 
 (5\tfrac{sd}{ds}-3 )_{((5))} \big(s^2U\big)
=
-(5(\tfrac{sd}{ds})-3) \big(6U^2-2U^3\big) 
+ 12 (\tfrac{sd}{ds}-1)\big(U-\tfrac{s}{3}\big)
 - 28 s^2 
\end{eqnarray}
where we use the notation 
$(5\tfrac{sd}{ds}-3 )_{((5))} = 
(5\tfrac{sd}{ds}-3 )(5\tfrac{sd}{ds}-5 )(5\tfrac{sd}{ds}-7 )(5\tfrac{sd}{ds}-9 )(5\tfrac{sd}{ds}-11)$.
\end{lemma}
\noindent Extracting the coefficient of $s^{g+2}$ in~\eqref{eq:tgdiff2}, we directly obtain that the constant $C'$ can be rewritten in the much simpler form:
$$
C' = \frac{4}{15} (5g+5)(5g+3)(5g+1)(5g-1) \tau_{g}.
$$

To sum up the present discussion, we have determined the first order asymptotic of the generating function $C_{g+2}^{(iii)}(z)$ of rooted maps counted by case (iii). Applying standard transfer theorems, the corresponding coefficient $c_{g+2}^{(iii)}(n):=[z^n]C_{g+2}^{(iii)}(z)$ satisfies:
\begin{eqnarray}\label{eq:cg3}
c_{g+2}^{(iii)}(n)\sim \frac{1}{15} 2^{6-5(g+2)} / \Gamma(\tfrac{5}{2}g-\tfrac{1}{2}) \tau_g \cdot n^{\frac{5}{2}(g+1)} 12^n,
\end{eqnarray}
where we have used that $(5g+5)(5g+3)(5g+1)(5g-1) \Gamma(\tfrac{5}{2}g-\tfrac{1}{2})=2^4\cdot \Gamma(\tfrac{5(g+2)}{2}-\tfrac{3}{2})$

\medskip
 It is now time to apply Miermont's bijection. If we disconnect the three endpoints belonging to the skeleton and the root vertex in a map from case (iii), we obtain a labelled map of genus $g$ with three faces, with one marked vertex inside each face, subject to the constraint that those three vertices have the same label (see Figure~\ref{fig:Tutte3}-Right). Miermont's bijection transforms this object into a bipartite quadrangulation of genus $g$ with \emph{six} marked vertices $(s_1,s_2,s_3,v_1,v_2,v_3)$, such that for each $i\in\{1,2,3\}$ the source $v_i$ is closer from the vertex $s_i$ than from the two other vertices $s_j$ (to see this, write precisely the inequalities analogue to~\eqref{eq:crossed} as in the previous section). Arguing as in the previous section (see the sketch of proof below), up to subdominating cases, this property  asymptotically characterizes those configurations, and we get:
\begin{lemma}\label{combiToMoment3points}
The number $c_{g+2}^{(iii)}(n)$ of configurations in case (iii) satisfies:
$$
\frac{c_{g+2}^{(iii)}(n)}{n^5/4 \cdot 2^{-2} m_g(n)} \sim \mathbf{E}[\YY_g^{(1:3)}\YY_g^{(2:3)}\YY_g^{(3:3)}].
$$
\end{lemma}
\noindent The reader can understand heuristically the denominator in the previous expression as follows. The factor $n^5/4$ comes from the fact that we have $\sim n^6$ ways to mark 6 vertices (among $n+2-2g$) but that the quadrangulation is unrooted so we divide by $4n$. The factor $2^{-2}$ corresponds to the fact that we have two parity constraints relating the distances of the six points together (these constraints enable us to choose the delays in such a way that the target vertices get the same label while respecting the parity constraints on delays required by Miermont's bijection). We only sketch the proof of the lemma, since it is similar to what we did in the previous section.
\begin{proof}[{Proof of Lemma~\ref{combiToMoment3points}} (sketch)]

First, $A(n):=2n \cdot c_{g+2}^{(iii)}(n)$ counts \emph{rooted} labelled three-face maps of genus $g$ with $n$ edges, with faces numbered $F_1,F_2,F_3$, with three marked vertices $v_1,v_2,v_3$ such that $v_i$ is incident to the face $F_i$ only, and such that $\ell(v_1)=\ell(v_2)=\ell(v_3)$. For $(\epsilon_2,\epsilon_3)\in\{0,1\}^2$, we introduce a variant $A^{\epsilon_2,\epsilon_3}(n)$ of this number, counting the same objects but where the last property is replaced by $\ell(v_1)=\ell(v_2)-\epsilon_2=\ell(v_3)-\epsilon_3$. For such an object we let $\delta_i$ be the minimum label in face $F_i$ for $i=1,2,3$, and we fix a translation class of labels by assuming that $\delta_1=0$. We also note $\epsilon_1:=0$.

Let $L$ be a three-face map counted by $A^{\epsilon_2,\epsilon_3}(n)$ and $Q$ be its associated quadrangulation by  Miermont's bijection. Then $Q$ is a bipartite quadrangulation of genus $g$ with $n$ faces, carrying three source vertices $s_1,s_2,s_3$ and the three marked vertices $v_1,v_2,v_3$, and is such that
\begin{eqnarray}\label{eq:constraints3}
\ell(v_i) = d(v_i,s_i)+\delta_i \leq d(v_i,s_j)+\delta_j, i\neq j.
\end{eqnarray}
Indeed this equation says that the minimum defining the label 
$\ell(v_i):=\min_{1\leq j\leq 3} d(v_i,s_j)+\delta_j$ in the Miermont labelling of the delayed quadrangulation $Q$ is reached by its $i$-th argument, which corresponds to the fact that vertex $v_i$ is incident to the face $F_i$ in $L$.
Writing that $\ell(v_i)-\epsilon_i =\ell(v_j)-\epsilon_j$ and applying~\eqref{eq:constraints3}, we find that :
\begin{eqnarray}\label{eq:voro3}
d(v_i,s_i)-\epsilon_i \leq d(v_j,s_i)-\epsilon_j.
\end{eqnarray}
Since $|\epsilon_i-\epsilon_j|\leq 2$ we can say, loosely speaking, that up to an error at most $2$, $s_i$ is closer from $v_i$ than from other $v_j$'s in $Q$.
Note that another constraint from Miermont's bijection is that, in $Q$, we have that $d(s_i,s_j)\equiv \delta_i+\delta_j\mod 2$
 for all $i,j$,
or equivalently, from ~\eqref{eq:constraints3}, that $d(v_i,v_j)\equiv \epsilon_i-\epsilon_j \mod 2$.

Conversely, let $B^{\epsilon_2,\epsilon_3}(n)$ be the number of rooted bipartite quadrangulations of genus $g$ with $n$ faces and six marked vertices $s_1,s_2,s_3,v_1,v_2,v_3$ such that we have, for all $i\neq j$:
\begin{eqnarray}\label{eq:constraints3Strong}
d(v_i,s_i) < d(v_j,s_i)-2
\end{eqnarray}
and such that $d(v_i,v_j)\equiv \epsilon_i-\epsilon_j\mod 2$.
Given such an object, defining for each vertex $\ell(v)=\min_{1\leq i \leq 3}d(s_i,v)+\delta_i$, where $\delta_i:= d(v_i,s_i)-\epsilon_i+\epsilon_1-d(v_1,s_1)$, we see from~\eqref{eq:constraints3Strong} that the minimum defining $\ell(v_i)$ is reached only by its $i$-th argument, and that $\ell(v_1)=\ell(v_2)-\epsilon_2=\ell(v_3)-\epsilon_3$. This ensures that the three-face map associated to such a quadrangulation by the 
 (reverse) Miermont bijection is one of the objects counted by $A^{\epsilon_2,\epsilon_3}(n)$. In fact, the converse is true up to asymptotically negligible terms. Indeed, configurations counted by $A^{\epsilon_1,\epsilon_2}(n)-B^{\epsilon_2,\epsilon_3}(n)$ correspond to cases where for at least one $i\neq j$, \eqref{eq:voro3} holds but \eqref{eq:constraints3Strong} does not: such configurations are few by Lemma~\ref{lemma:zeroProba}(ii). Details are similar to the proof of Lemma~\ref{lemma:GeoToVor} and  we  obtain that 
$$B^{\epsilon_2,\epsilon_3}(n) = 2 A^{\epsilon_1,\epsilon_2}(n) + o(n^6 m_g(n)),$$ 
where the factor of $2$ comes from the 2-to-1 nature of Miermont's bijection.

Now, similarly as in Lemma~\ref{lemma:equalContrib}, it is easy to see that for different $\epsilon_1,\epsilon_2$ the numbers $A^{\epsilon_1,\epsilon_2}(n)$ have the same first order contribution and:
$$
\sum_{(\epsilon_1,\epsilon_2)\in\{0,1\}^2} A^{\epsilon_1,\epsilon_2}(n) \sim 4 \cdot A^{0,0}(n). 
$$
Now $\sum_{(\epsilon_1,\epsilon_2)\in\{0,1\}^2} B^{\epsilon_2,\epsilon_3}(n) $ counts rooted bipartite quadrangulations with marked vertices $v_1,v_2,v_3,$ $s_1,s_2,s_3$ such that \eqref{eq:constraints3Strong} holds, and in which no parity constraints remain. Recalling that $c_{g+2}^{(iii)}(n)=\tfrac{1}{2n}A^{0,0}(n)$ we finally get 
$$
c_{g+2}^{(iii)}(n) \sim \frac{1}{8\cdot 2n} n^6 m_g(n) 
\cdot \mathbf{P}\big\{\forall\, 1\leq i\neq j \leq 3:\, d(\vv_i,\sst_i) < d(\vv_j,\sst_i)-2\big\}
$$
where the probability is over $\qgn\in_u\QQgn$ and six uniform independent vertices $(\vv_1,\vv_2,\vv_3,\sst_1,\sst_2,\sst_3)$ in $\qgn$. Finally, arguing exactly as in the proof of Lemma~\ref{lemma:convergenceCell}, the last probability converges to $\mathbf{E}[\YY_g^{(1:3)}\YY_g^{(2:3)}\YY_g^{(3:3)}]$ and we are done. 
\end{proof}

From the last lemma and the expansion of $m_g(n)$ it follows that 
$$
c_{g+2}^{(iii)}(n) \sim 2^{-2-5g}/\Gamma(\tfrac{5}{2}g-\tfrac{1}{2})\tau_g
\cdot \mathbf{E}[\YY_g^{(1:3)}\YY_g^{(2:3)}\YY_g^{(3:3)}]
 \cdot n^{5+\frac{5}{2}(g-1)} 12^n.
$$
Comparing with the previously obtained expansion~\eqref{eq:cg3} of $c_{g+2}^{(iii)}(n)$ we find:
$$
\mathbf{E}[\YY_g^{(1:3)}\YY_g^{(2:3)}\YY_g^{(3:3)}] \sim \frac{1}{15 \cdot 2^2} =\frac{1}{60}
$$
as claimed!

It only remains to prove Lemma~\ref{lemma:eliminate}:
\begin{proof}[Proof of Lemma~\ref{lemma:eliminate}]
	We give a simple proof based on linear algebra, relatively brutal and (therefore) easily computerized. Let $E_0$ denote Equation~\eqref{eq:tgdiff}, and consider its derivatives $E_i:=(\tfrac{d}{ds})^i E_0$ for $i=1,2,3$. We thus obtain a polynomial system of four equations $\{E_i,0\leq i\leq 3\}$, involving the quantities $U_i=(\tfrac{d}{ds})^i U(s)$ for $0\leq i \leq 5$.  This system is \emph{linear} and \emph{triangular} in $U_2,U_3,U_4,U_5$ (note that it is \emph{not} linear in $U_0$ and $U_1$). We can then solve for these four quantities and we obtain an expression of each $U_i$ for $2 \leq i \leq 5$ as a (nonlinear) polynomial of $U_0=U(s)$ and $U_1=\tfrac{d}{ds}U(s)$ (with coefficients that are Laurent polynomials of $s$). 
Now expand the quantity $(5\tfrac{sd}{ds}-3 )_{((5))} \big(s^2U\big)$ as a linear combination of the $U_i$'s, and substitute the expressions just obtained of  $U_i$ for $i\geq 2$ in it.
We obtain an equation of the form:
$$
(5\tfrac{sd}{ds}-3 )_{((5))} \big(s^2U\big) = \mbox{Polynomial}(s,s^{-1},U(s),\tfrac{d}{ds}U(s))
$$
which, computations made, is~\eqref{eq:tgdiff2}.
\end{proof}

\section{A comment on nonorientable surfaces} An anonymous referee asked us to discuss the case of nonorientable surfaces. The  $t_g$-recurrence and Miermont's bijection both have analogues for nonorientable surfaces, respectively \emph{the $p_g$-recurrence}~\cite{Marino, Carrell}, and the bijection of~\cite{ChapuyDolega}. The construction of Brownian nonorientable surfaces has not been done yet, but this does not seem to be an intrinsic obstruction. A more serious problem to generalize our work is the fact that, when writing the analogue of Equation~\eqref{eq:Tutte} an additional term appears, corresponding to the case where the root edge is a ``twisted'' non-isthmic edge of the skeleton. This term is naturally expressed in terms of  the generating function of l.1.f.m. of genus $g+\frac{1}{2}$ with two marked corners in the unique face, whose labels differ by at most one. From there, in analogy with Theorem~\ref{thm:obs1}, it should be possible to write a recurrence for the numbers $p_g$, $t_g$ that involves, in addition to the second moment of the variable $\mathbf{X}_g$, the first moment of a random variable $\mathbf{Y}_{g+1/2}$ that would count, in some sense, ``the renormalized number of pairs of points at equal distance from the root in a large random quadrangulation of genus $g+\frac{1}{2}$''. 
We do not pursue here the task of writing this recurrence precisely nor of examining if any phenomenon arises that would enable one to conjecture a remarkable value for this first moment.
We conclude by mentioning the recent results of~\cite{ABACFG}. Motivated by the present paper, it is shown there  that the analogue of Conjecture~\ref{conjecture} is true for random one-face maps \emph{on any fixed surface}, orientable or not. Random one-face maps belong to a very different universality class than random quadrangulations, yet the fact that the random partition of mass induced by Vorono\"i tessellations is independent of the surface is as surprising as in the case studied here. This suggests that, if true,  our main conjecture should hold for nonorientable Brownian surfaces as well. Further evidence for this guess is that the analogue of Theorem~\ref{cor:main} holds for the projective plane, namely $\mathbf{E}\mathbf{X}_{1/2}(1-\mathbf{X}_{1/2})=\frac{1}{6}$ with obvious notation. We have checked this value using the bijection of~\cite{ChapuyDolega} that relates it to the count of certain labelled two-face maps on the projective plane, in the same way as in Lemma~\ref{lemma:convergenceCell}. In order to estimate the number of configurations we used a scheme decomposition and the methods of~\cite{CMS}. The computation, which is computer assisted, involves $11\cdot 5!$ different cases (there are 11 unrooted $2$-face maps on the projective plane with four trivalent vertices and two leaves with one leaf in each face. For each of them one has to consider the $5!$ different vertex labellings of the graph in which the two leaves are identified, in order to compute the corresponding contribution).

\bibliographystyle{alpha}
\bibliography{biblio}
\end{document}